\documentclass[a4paper,11pt,reqno]{amsart}
\usepackage{amssymb}
\usepackage{amsmath}
\usepackage{ifthen}
\usepackage{enumitem}
\usepackage[dvips]{graphicx}
\usepackage[a4paper,margin=0.85in]{geometry}
\nonstopmode \numberwithin{equation}{section}
\usepackage{amssymb}
\usepackage{ifthen}

\usepackage{amsmath}
\usepackage[T1]{fontenc} %skandit
\usepackage{comment}

\nonstopmode \numberwithin{equation}{section}
\theoremstyle{plain}
\newtheorem{thm}{Theorem}
\numberwithin{thm}{section}
\newtheorem{cor}{Corollary}
\numberwithin{cor}{section}
\newtheorem{lem}{Lemma}
\numberwithin{lem}{section}
\newtheorem{prop}{Proposition}

\newtheorem{conj}{Conjecture}

\theoremstyle{definition}
\newtheorem{defn}{Definition}[section]

\newtheorem{prob}{Problem}
\newtheorem{rem}{Remark}[section]

\newcounter{minutes}
\divide\time by 60
\newcounter{hours}
\multiply\time by 60
\addtocounter{minutes}{-\time}
\newcounter {own}
\def\theown {\thesection       .\arabic{own}}

\newenvironment{pf}[1][]{%
	\vskip 3mm
	\noindent
	\ifthenelse{\equal{#1}{}}%
	{{\slshape Proof. }}%
	{{\slshape #1.} }%
}%
{\qed\bigskip}

\theoremstyle{plain}
\newtheorem{Thm}{Theorem}

\numberwithin{equation}{section}
\def\be{\begin{equation}}
	\def\ee{\end{equation}}

\newcommand{\bee}{\begin{enumerate}}
	\newcommand{\eee}{\end{enumerate}}

\newcommand{\blem}{\begin{lem}}
	\newcommand{\elem}{\end{lem}}
\newcommand{\bthm}{\begin{thm}}
	\newcommand{\ethm}{\end{thm}}
\newcommand{\bcor}{\begin{cor}}
	\newcommand{\ecor}{\end{cor}}
\newcommand{\beg}{\begin{examp}}
	\newcommand{\eeg}{\end{examp}}
\newcommand{\begs}{\begin{examples}}
	\newcommand{\eegs}{\end{examples}}
\allowdisplaybreaks
\newcommand{\bdefn}{\begin{defn}}
	\newcommand{\edefn}{\end{defn}}

\newcommand{\bprob}{\begin{prob}}
	\newcommand{\eprob}{\end{prob}}
\newcommand{\bei}{\begin{itemize}}
	\newcommand{\eei}{\end{itemize}}

\newcommand{\bcon}{\begin{conj}}
	\newcommand{\econ}{\end{conj}}
\newcommand{\bcons}{\begin{conjs}}
	\newcommand{\econs}{\end{conjs}}
\newcommand{\bprop}{\begin{prop}}
	\newcommand{\eprop}{\end{prop}}
\newcommand{\br}{\begin{rem}}
	\newcommand{\er}{\end{rem}}
\newcommand{\brs}{\begin{rems}}
	\newcommand{\ers}{\end{rems}}
\newcommand{\bo}{\begin{obser}}
	\newcommand{\eo}{\end{obser}}
\newcommand{\bos}{\begin{obsers}}
	\newcommand{\eos}{\end{obsers}}
\newcommand{\bpf}{\begin{pf}}
	\newcommand{\epf}{\end{pf}}
\newcommand{\ba}{\begin{array}}
	\newcommand{\ea}{\end{array}}
\newcommand{\beq}{\begin{eqnarray}}
	\newcommand{\beqq}{\begin{eqnarray*}}
		\newcommand{\eeq}{\end{eqnarray}}
	\newcommand{\eeqq}{\end{eqnarray*}}

\begin{document}
	\title{Riesz Theorem and Riesz-Fej\'er inequality for weighted harmonic Bergman spaces with applications to M\"obius invariant spaces}

	\author{Himadri Halder}
	\address{Himadri Halder,
		Department of Mathematics and Computing,
		Indian Institute of Technology (ISM) Dhanbad,
		Dhanbad-826004, Jharkhand, India.}
	\email{himadrihalder@iitism.ac.in}
	
	\author{Rohit Kumar}
	\address{Rohit Kumar,
		Department of Mathematics,
		Indraprastha Institute of Information Technology, Delhi,
		New Delhi-110020, India.}
	\email{rohitk12798@gmail.com}

	\subjclass[{AMS} Subject Classification:]{Primary 31A05, 30H20, 30C62}
	
%	\def\thefootnote{}
%	\footnotetext{ {\tiny File:~\jobname.tex,
%			printed: \number\year-\number\month-\number\day,
%			\thehours.\ifnum\theminutes<10{0}\fi\theminutes }
%	} \makeatletter\def\thefootnote{\@arabic\c@footnote}\makeatother

	\begin{abstract}
	The aim of this paper is twofold. First, we establish a Riesz conjugate theorem for weighted harmonic Bergman spaces. More precisely, we prove that if $f=u+iv$ is a harmonic $K$-quasiregular mapping in $\mathbb{D}$ and the real part $u$ belongs to the weighted harmonic Bergman space $a_\alpha^p$, $0<p<\infty$, then the imaginary part $v$ also belongs to the same space, together with a quantitative norm estimate. Moreover, for $1<p<\infty$, the corresponding constant is shown to be independent of the weight parameter $\alpha$. Second, we establish Riesz--Fej\'er inequalities for weighted harmonic Bergman spaces for $1<p<\infty$.
	In the special case $p=2$, we further improve the corresponding constant by using the Hilbert space structure and orthogonality techniques. As applications of our main results, we establish Riesz conjugate theorems and Riesz--Fej\'er inequalities for the M\"obius invariant spaces $Q(n,p,\alpha)$ introduced by Zhu [Illinois J. Math. 51 (2007), pp. 977--1002] and their harmonic counterparts $Q_h(n,p,\alpha)$ introduced by Sun, Liu, and Wang [Potential Anal. 65 (2026), Article no. 12].		
	\end{abstract}
	
	\maketitle
	\pagestyle{myheadings}
	\markboth{Himadri Halder and  Rohit Kumar}{Riesz Theorem and Riesz--Fej\'er inequality for weighted harmonic Bergman space}
	
	\section{\textbf{Introduction and Main Results}}
	Let $\mathbb{D} = \{z \in \mathbb{C} : |z| < 1\}$ denote the open unit disk in the complex plane, bounded by the unit circle $\mathbb{T} = \{z \in \mathbb{C} : |z| = 1\}$. For $0 < p < \infty$, the growth and boundary behavior of a complex-valued function $f$ defined on $\mathbb{D}$ are classically characterized by its $p$-th integral mean, defined for $0 \le r <1$ by
	$$M_p(r, f) := \left( \frac{1}{2\pi} \int_{0}^{2\pi} |f(re^{i\theta})|^p \, d\theta \right)^{1/p}.$$	
	 For $0 < p < \infty$, an analytic function $f$ in $\mathbb{D}$ belongs to the Hardy space $H^p$ if its $p$-th integral mean remains uniformly bounded as it approaches the boundary. That is, $$\|f\|_{H^p} := \lim_{r \to 1^-} M_p(r, f) = \sup_{0 < r < 1} M_p(r, f) < \infty$$
	 for every $f \in H^{p}$.
	 \vspace{1mm}
	 
	 Let $dA(z)$ denote the standard area measure on $\mathbb{D}$, normalized so that the total area of the disk is unity, defined by $$dA(z) = \frac{dx\,dy}{\pi} = \frac{1}{\pi} r\,dr\,d\theta, \quad z = x+iy = re^{i\theta}.$$For parameters $-1 < \alpha < \infty$, the standard weighted area measure on $\mathbb{D}$ is defined by $dA_\alpha(z) = (\alpha + 1)(1 - |z|^2)^\alpha \, dA(z)$. 
	 \vspace{1mm}
	 
	 When $0 < p < \infty$, an analytic function $f$ in $\mathbb{D}$ belongs to the weighted Bergman space $A_\alpha^p$ if it satisfies
	 $$\|f\|_{A_\alpha^p} := \left( \int_{\mathbb{D}} |f(z)|^p  \, dA_{\alpha}(z) \right)^{1/p} < \infty.$$
	 In particular, for $\alpha=0$, this simplifies to the standard Bergman space $A^p$. 
	 \vspace{1mm}
	 
	 It is well known that $H^p \subset A^p$. We refer the reader to Duren's monograph \cite{Duren-Hp-book} for a comprehensive treatment of Hardy spaces, and to the books by Duren \cite{Duren-Schuster-AMS-2004} and Hedenmalm, Korenblum, and Zhu \cite{Hedenmalm-Korenblum-Zhu-GTM} for the theory of Bergman spaces.
	 \vspace{1mm}
	 
	 A complex-valued function $f = u + iv$ on $\mathbb{D}$ is harmonic if its real and imaginary parts are real-valued harmonic functions. The harmonic Hardy space $h^p$ consists of all complex-valued harmonic functions $f$ on $\mathbb{D}$ satisfying $$\|f\|_{h^p}: = \sup_{0 < r < 1} M_p(r, f) < \infty.$$
	 For $-1<\alpha<\infty$ and $0<p<\infty$, the weighted harmonic Bergman space $a_\alpha^p$ consists of all complex-valued harmonic functions $f$ on $\mathbb{D}$ for which the following integral converges $$\|f\|_{a_\alpha^p} := \left( \int_{\mathbb{D}} |f(z)|^p  \, dA_{\alpha}(z) \right)^{1/p} < \infty.$$
	 For $\alpha=0$, we simply write $a_0^p := a^p$, the usual harmonic Bergman space. Furthermore, one has $h^p \subset a^p$.
	 \vspace{1mm}
	 
	The relationship between a real-valued harmonic function $u$ and its unique normalized harmonic conjugate $v$, with $v(0)=0$, is fundamental in geometric function theory. Formulated as an operator theory problem, we ask whether a target harmonic space is invariant under the conjugation operator $T: u \mapsto v$. To facilitate the discussion, we first consider the following spaces of real-valued harmonic functions: 
	$$
	\mathcal{H}_h^{p}(\mathbb{D};\mathbb{R})
	=
	\left\{
	u:\mathbb{D}\to\mathbb{R} :
	u \text{ is harmonic and }
	\|u\|_{h^{p}}<\infty
	\right\}
	$$
	 and 
	 $$
	 \mathcal{A}_{h,\alpha}^{p}(\mathbb{D};\mathbb{R})
	 =
	 \left\{
	 u:\mathbb{D}\to\mathbb{R} :
	 u \text{ is harmonic and }
	 \|u\|_{a_\alpha^{p}}<\infty
	 \right\}.
	 $$
	 Set $\mathcal{A}_{h}^{p}(\mathbb{D};\mathbb{R}):=\mathcal{A}_{h,0}^{p}(\mathbb{D};\mathbb{R})$. Then both $\mathcal{H}_h^{p}(\mathbb{D};\mathbb{R})$ and $\mathcal{A}_{h,\alpha}^{p}(\mathbb{D};\mathbb{R})$ are naturally real Banach spaces for $1 \leq p < \infty$ and real quasi-Banach spaces for $0<p<1$. Since $\mathbb{D}$ is simply connected, every function
	 $u \in \mathcal{H}_h^{p}(\mathbb{D};\mathbb{R})\, (\mbox{or}\, \mathcal{A}_{h,\alpha}^{p}(\mathbb{D};\mathbb{R}))$
	 admits a unique harmonic conjugate $v$, normalized by
	 $v(0)=0$. Thus, the harmonic conjugation operator
	 \[
	 T:\mathcal{H}_h^{p}(\mathbb{D};\mathbb{R})\, \, (\mbox{or}\, \mathcal{A}_{h,\alpha}^{p}(\mathbb{D};\mathbb{R}))
	 \longrightarrow
	 \mathcal{H}_h^{p}(\mathbb{D};\mathbb{R})\, (\mbox{or} \, \, \mathcal{A}_{h,\alpha}^{p}(\mathbb{D};\mathbb{R}))
	 \quad \mbox{by} \,\,
	 T(u)=v,
	 \]
	 is well defined. Moreover, $T$ is a real-linear operator.
	 \vspace{1mm}
	 
	 For Hardy spaces, the operator $T$ is famously bounded by the well-known Riesz conjugate theorem. 
	 
	\begin{Thm}\cite[Theorem 4.1]{Duren-Hp-book}\label{thm-A}
	Let $1 < p < \infty$. If $u \in \mathcal{H}_h^{p}(\mathbb{D};\mathbb{R})$, then its harmonic conjugate $v$ belongs to $\mathcal{H}_h^{p}(\mathbb{D};\mathbb{R})$. There exists a constant $A_p > 0$ such that:
	$$M_p(r, v) \le A_p M_p(r, u).$$	
	\end{Thm}
	Theorem \ref{thm-A} breaks down completely at the endpoints $p=1$ and $p=\infty$ (see \cite[p. 56]{Duren-Hp-book}). A partial substitute was later obtained by Kolmogorov, who proved that if $u\in \mathcal{H}_h^{1}(\mathbb{D};\mathbb{R})$, then its harmonic conjugate $v$ belongs to $\mathcal{H}_h^{p}(\mathbb{D};\mathbb{R})$ for every $0<p<1$, although $v$ need not belong to $\mathcal{H}_h^{1}(\mathbb{D};\mathbb{R})$. The situation is quite different when $0<p<1$. Indeed, Hardy and Littlewood  \cite{Hardy-Littlewood-JRAM-1932} showed that, in general, the harmonic conjugate of a function in
	$\mathcal{H}_h^{p}(\mathbb{D};\mathbb{R})$ need not belong to $\mathcal{H}_h^{q}(\mathbb{D};\mathbb{R})$ for any $q>0$. Equivalently, the harmonic conjugation operator $T$ does not map $\mathcal{H}_h^{p}(\mathbb{D};\mathbb{R})$ into
	$\mathcal{H}_h^{q}(\mathbb{D};\mathbb{R})$ for any $q>0$. 
	\vspace{1mm}
	
In contrast, Hardy and Littlewood \cite{Hardy-Littlewood-JRAM-1932} showed that area-based integration has a smoothing effect, making the harmonic conjugation operator well behaved even for smaller values of $p$. 

\begin{Thm} \cite[Theorem 5]{Hardy-Littlewood-JRAM-1932} \label{thm-B}
	Let $0 < p < \infty$. If $u$ belongs to the Bergman space $\mathcal{A}_{h}^{p}(\mathbb{D};\mathbb{R})$, then its harmonic conjugate $v$ must also belong to $\mathcal{A}_{h}^{p}(\mathbb{D};\mathbb{R})$, satisfying $\|v\|_{a^p} \le C_p \|u\|_{a^p}$ for some constant $C_p>0$.
\end{Thm}
Thus, unlike the Hardy space case, Theorem \ref{thm-B} establishes that the conjugation operator $T$ is bounded on weighted Bergman space $\mathcal{A}_{h}^{p}(\mathbb{D};\mathbb{R})$ for all $0<p<\infty$. It is worth mentioning that, for the classical weighted Bergman spaces $\mathcal{A}_{h,\alpha}^{p}(\mathbb{D};\mathbb{R})$ the boundedness of the harmonic conjugation operator is an immediate consequence of \cite[Corollary 6]{Pelaez-Rattya-AMP-2020}.
\vspace{1mm}
\subsection{Riesz conjugate theorem for weighted harmonic Bergman space}
One may naturally ask whether Theorems \ref{thm-A} and \ref{thm-B} remain valid for complex-valued harmonic functions. More precisely, let $f=u+iv$ be a complex-valued harmonic function  in $\mathbb{D}$ with real part $u \in h^p$. Does it follow that $v \in h^p$ for $1<p<\infty$? Likewise, if $u \in a^p$, does one necessarily have $v \in a^p$ for all $0<p<\infty$. In general, the answer to both questions is negative (see \cite{Astala-Koskela-PAMQ-2011,Chen-Huang-Wang-Xiao-arXiv-2025,Das-Rasila-PA-2026,Liu-Zhu-AdvMath-2023}). For the Hardy space setting, Liu and Zhu \cite{Liu-Zhu-AdvMath-2023} showed that the Riesz conjugate theorem holds provided that the harmonic mapping $f$ is quasiregular. This result was subsequently sharpened by Chen {\it et al.} \cite{Chen-Huang-Wang-Xiao-arXiv-2025}. For the Bergman space setting $a^p$, an analogous sufficient condition was recently obtained by Das and Rasila \cite{Das-Rasila-PA-2026}. We next recall the notions of sense-preserving harmonic mappings, quasiconformal mappings, and quasiregular mappings.
\vspace{1mm}

Every complex-valued harmonic function in $\mathbb{D}$ admits a canonical, unique decomposition of the form $f = h + \overline{g}$, where $h$ and $g$ are analytic in $\mathbb{D}$ and are uniquely determined by the normalization $g(0) = 0$. Such functions $f$ in $\mathbb{D}$ is said to be {\it sense-preserving} if its Jacobian
$J_f(z):=|h'(z)|^2-|g'(z)|^2$ is positive in $\mathbb{D}$, or equivalently, $|g'(z)|<|h'(z)|$ for all $z\in\mathbb{D}$.
A sense-preserving harmonic function $f=h+\overline{g}$ is called
\emph{$K$-quasiregular}, $K\ge1$, if
\[
\frac{|h'(z)|+|g'(z)|}{|h'(z)|-|g'(z)|}\le K,
\qquad z\in\mathbb{D},
\]
or equivalently,
\[
|g'(z)|\le k|h'(z)|,
\qquad z\in\mathbb{D},
\]
where
\[
k:=\frac{K-1}{K+1}<1.
\]

\noindent A harmonic function is called \emph{K-quasiconformal} if it is $K$-quasiregular and homeomorphic in $\mathbb{D}$. For a detailed account of the $H^p$-theory of quasiconformal mappings, see Astala and Koskela \cite{Astala-Koskela-PAMQ-2011}.
\vspace{1mm}

We are now in a position to state the result of Liu and Zhu from \cite{Liu-Zhu-AdvMath-2023} regarding harmonic quasiregular function. It says that if $f=u+iv$ is a $K$-quasiregular harmonic function in $\mathbb{D}$ such that $v(0)=0$ with real part $u \in h^p$, then $v \in h^p$ for $1<p<\infty$, satisfying $M_p(r, v) \le A_{p,K} M_p(r, u)$ for some positive constant $A_{p,K}$ depends only on $p$ and $K$. Very recently, Chen and Kalaj \cite{Chen-Kalaj-k-k'-2025} further sharpened these results by establishing conjugate-type properties for harmonic $(K,K')$-quasiregular mappings. Earlier, Kalaj \cite{Kalaj-JMAA-2025} obtained Kolmogorov-type theorems for harmonic quasiregular mappings, and subsequently established a quasiregular analogue of Zygmund's theorem \cite{Kalaj-AMP-2025}. For recent advances on the Riesz theorem and related results concerning quasiregular mappings, we refer the reader to \cite{Kalaj-Mateljevic-PA-2012,Kalaj-TAMS-2019,Kalaj-PA-2025,Melentijevic-Markovic-PA-2023}.
\vspace{1mm}

Turning now to the Bergman space setting, motivated by the work of Liu and Zhu on harmonic Hardy spaces, Das and Rasila recently investigated the Riesz conjugate theorem for harmonic Bergman spaces in \cite[Theorem 1]{Das-Rasila-PA-2026}. Their result reads as follows. Let $0<p<\infty$, and let $f=u+iv$ be a harmonic $K$-quasiregular
mapping in $\mathbb{D}$. If $u\in a^p$, then $v\in a^{p}$. Moreover, there exists a constant $C_{p,K}>0$, depending only on $p$ and $K$, such that $\|v\|_{a^p}\le C_{p,K}\|u\|_{a^p}$.
%\begin{Thm} \label{thm-C} \cite[Theorem 1]{Das-Rasila-PA-2026}
%	Let $0<p<\infty$, and let $f=u+iv$ be a harmonic $K$-quasiregular
%	mapping in $\mathbb{D}$. If	$u\in a^p$, then $v\in a^{p}$. Moreover, there 	exists a constant $C_{p,K}>0$, depending only on $p$ and $K$, such that 	\[	\|v\|_{a^p}\le C_{p,K}\|u\|_{a^p}.	\]
%\end{Thm}

To the best of our knowledge, the Riesz conjugate theorem for weighted harmonic Bergman spaces has not been investigated in the existing literature. One of the main objectives of this paper is to fill this gap. Our first result, stated below, establishes the Riesz conjugate theorem for weighted harmonic Bergman spaces.
\begin{thm} \label{thm-1.1}
	Suppose $-1<\alpha<\infty$ and $0 < p < \infty$. Let $f = u + iv$ be a harmonic $K$-quasiregular function in $\mathbb{D}$. If the real part $u$ belongs to the weighted harmonic Bergman space $a^p_{\alpha}$, then the imaginary part $v$ also belongs to $a^p_{\alpha}$. Furthermore, there exists a positive constant $C_{p,K,\alpha}$ depending only on $p,K$ and $\alpha$ such that
	$$\|v\|_{{a^p_\alpha}} \le C_{p,K,\alpha} \|u\|_{{a^p_\alpha}}.
	$$ 
	In particular, when $p>1$, $C_{p,K,\alpha}$ is independent of $\alpha$.
\end{thm}

\subsection{Riesz-Fej\'er inequality for weighted harmonic Bergman space}
 Our second aim of this article is a parallel line of research relates a function's boundary behavior to its internal geometry, rooted in the classical Riesz-Fej\'{e}r inequality. The Riesz--Fej\'er inequality is one of the fundamental results in the theory of Hardy
 spaces and geometric function theory. It provides a fundamental comparison
 between the integral of an analytic function along a diameter of the unit disk
 and its boundary integral over the unit circle.  
 The inequality was first established by Riesz and Fej\'er in 1928.
 \begin{Thm}[Riesz--Fej\'er {\cite[Theorem 3.13]{Duren-Hp-book}}] \label{Thm-C}
 	Let $0<p<\infty$. If $f\in H^{p}$, then
 	\[
 	\int_{-1}^{1}|f(x)|^{p}\,dx
 	\le
 	\frac12\int_{0}^{2\pi}|f(e^{i\theta})|^{p}\,d\theta.
 	\]
 	Moreover, the constant $1/2$ is best possible.
 \end{Thm}This result has become a classical
 tool in Hardy space theory and has stimulated extensive research over the past
 century (see \cite{Andreev-Palermo-2012,Beckenbach-JLMS-1938,Calderon-PAMS-1950,DuPlessis-JLMS-1955,Duren-Hp-book,Frazer-JLMS-1934,Huber-AnnMath-1956,Kayumov-Ponnusamy-Kaliraj-PA-2020,Melentijevic-Bozin-PA-2021}). Besides its intrinsic analytic
 interest, the inequality admits an elegant geometric interpretation: if the unit
 disk is mapped conformally onto a rectifiable Jordan domain, then the image of any diameter has length at most one-half of the boundary length of the image
 domain. Consequently, the Riesz--Fej\'er inequality has found numerous
 applications in complex analysis, conformal mapping, coefficient estimates,
 integral means, and related extremal problems (see \cite{Kalaj-Mateljevic-PA-2012,Kayumov-Ponnusamy-Kaliraj-PA-2020,Melentijevic-Bozin-PA-2021}). 
 
 A notable refinement was obtained by Frazer \cite{Frazer-JLMS-1934} in 1934, who replaced a single
 diameter by an arbitrary pair of diameters forming an acute angle. His theorem
 revealed that the geometry of the diameters influences the optimal constant,
 thereby extending the classical Riesz--Fej\'er inequality in a natural geometric
 direction.
 
 Subsequently, Beckenbach \cite{Beckenbach-JLMS-1938} considerably generalized the Riesz--Fej\'er
 inequality by replacing $|f|^{p}$ with positive log-subharmonic functions.
 Further extensions under weaker regularity assumptions, as well as
 generalizations to various function spaces, were later obtained by several
 authors, demonstrating the broad applicability of the inequality (see \cite{Calderon-PAMS-1950,Huber-AnnMath-1956}).
 
 An important breakthrough in the Bergman space setting was achieved by Zhu \cite{Zhu-Monthly-2004},
 who developed a systematic correspondence between Hardy and weighted
 Bergman spaces. 
 The following theorem, due to Zhu \cite{Zhu-Monthly-2004}, extends the
 classical Riesz--Fej\'er inequality to weighted Bergman spaces.
 \begin{Thm}\cite[Theorem 4]{Zhu-Monthly-2004} \label{thm-D}
 	Let $0<p<\infty$ and $\alpha>-1$. If
 	$f\in A_{\alpha}^{p}$, then
 	\[
 	\int_{-1}^{1}(1-|x|)^{\alpha+1}|f(x)|^{p}\,dx
 	\le
 	\pi\, \, \int_{\mathbb D}|f(z)|^{p}\,dA_{\alpha}(z).
 	\]
 \end{Thm}
 
 As an application of this translation principle, he established
 weighted Bergman analogues of several classical Hardy space inequalities,
 including the Riesz--Fej\'er inequality, Hardy's inequality, and the
 Hardy--Littlewood inequality. This work revealed that many fundamental
 Hardy space inequalities possess natural counterparts in weighted Bergman
 spaces.
 
 More recently, attention has shifted towards harmonic mappings. A sharp harmonic analogue of the classical Riesz--Fej\'er inequality for harmonic Hardy spaces was first established by Kayumov \emph{et al.} \cite{Kayumov-Ponnusamy-Kaliraj-PA-2020} for $1<p\le2$, while the remaining case $p>2$ was later settled by Melentijevi\'c and Bo\v{z}in \cite{Melentijevic-Bozin-PA-2021}.
 
 Motivated by Frazer's theorem, Das and Kaliraj \cite{Das-Kaliraj-JMAA-2022} recently established a Riesz--Fej\'er type inequality for complex-valued harmonic functions by estimating the integral of $|f|^{p}$ over the union of two diameters of the unit disk in terms of the boundary integral over the unit circle. Recent developments on higher-dimensional Riesz--Fej\'er inequalities can be found in the work of Chen and Hamada \cite{Chen-Hamada-MZ-2023}.
 
 %Their work also produced several  Hilbert-type inequalities for real sequences and further highlighted the close  connection between harmonic mappings and classical integral inequalities.
 
 Despite these significant developments, no analogue of the Riesz--Fej\'er
 inequality has been established for harmonic Bergman spaces. Since Bergman spaces constitute one of the principal settings of modern
 complex analysis, it is natural to ask whether Theorem \ref{thm-D} admits a harmonic counterpart. One of the primary objectives of the present paper is to answer this question. We establish Riesz--Fej\'er type inequalities for weighted harmonic Bergman spaces, providing a harmonic counterpart to Zhu's result for analytic weighted Bergman spaces.
 % shifting the classical boundary-integral controls to full-domain area bounds.
 \begin{thm} \label{thm-1.2}
 	Let $1<p<\infty$ and $-1 < \alpha < \infty$. If $f = h + \bar{g}$ is a complex-valued harmonic function in $a_{\alpha}^p$, then 
 	$$
 	\int_{-1}^{1} (1 - |x|)^{\alpha+1} |f(x)|^p \, dx \leq  \pi \, \sec^{p}\!\left(\dfrac{\pi}{2p}\right)\, \int_{\mathbb{D}} |f(z)|^p \, dA_{\alpha}(z).
 	$$
 \end{thm}
 Although Theorem \ref{thm-1.2} holds for all $1<p<\infty$, the Hilbert space case $p=2$ deserves separate investigation because of its substantially richer structure. In contrast to the general Banach space setting, weighted harmonic Bergman spaces become Hilbert spaces when $p=2$, making available powerful tools such as orthogonality, reproducing kernels, discrete Schur estimates and interpolation techniques. Consequently, the $L^{2}$ theory often yields considerably sharper estimates together with applications that are unavailable for general values of $p$.
 
 This phenomenon has already been observed in the analytic setting. Jakobczak \cite{Jakobczak-Palermo-2008} initiated the study of radial integrability of functions in weighted Bergman spaces in the $L^{2}$ framework, while Andreev \cite{Andreev-Palermo-2012} later established sharp Riesz-Fej\'er type inequalities with explicit constants by combining discrete Schur estimates and the Stein--Weiss interpolation theorem. 
 \begin{Thm}\cite[Theorem 1]{Andreev-Palermo-2012} \label{thm-E}
 	Let $f\in A_{\alpha}^{2}$. 	Let
 	\[
 	\Lambda_{\alpha}(f)
 	:=\int_{0}^{1}|f(\zeta x)|^{2}(1-x)^{1+\alpha}\,dx,
 	\qquad |\zeta|=1.
 	\]
 	Then $\Lambda_{\alpha}(f) \leq \lambda_\alpha \, \|f\|_{A_{\alpha}^{2}}^{2}$. If $-1\leq\alpha\leq0$, then $\lambda_{\alpha}	\leq \pi^{-\alpha}$, whereas $\lambda_{\alpha}
 	\leq 	\frac{1}{1+\alpha}$ for $\alpha\geq0$.
 \end{Thm}
 Andreev showed that these inequalities are closely connected with the boundedness and compactness of certain Toeplitz operators, thereby revealing an important operator-theoretic aspect of the Hilbert space theory. 
 %More recently, Kasuga extended Andreev's inequality to weighted Bergman spaces $A_{\alpha}^{p}$ for arbitrary $0<p<\infty$. However, the proof relies essentially on Andreev's $L^{2}$ theorem and does not recover the Hilbert-space structure or its operator-theoretic consequences.
 
 Motivated by these developments, it is natural to ask whether the better $L^{2}$ phenomenon persists for weighted harmonic Bergman spaces. Unlike analytic functions, harmonic mappings contain both analytic and co-analytic parts, and therefore the arguments used in the analytic Bergman space cannot be transferred directly. Thus, establishing an analogue of Theorem \ref{thm-E} for weighted harmonic Bergman spaces is not merely a specialization of the general $p$-th power but a problem of independent interest.
 
 In this paper, we answer this question by establishing the harmonic counterpart of Theorem \ref{thm-E}. We further show that, in the Hilbert space setting $p=2$, the constant in Theorem \ref{thm-1.2} admits an improvement through the use of orthogonality.
% We assume throughout the validity of the following analytic radial trace inequality (Kasuga-type inequality): there exists a constant $\lambda_\alpha > 0$ such that for every $f \in A_\alpha^2$,
% $$\int_0^1 |f(x)|^2 (1 - x)^{\alpha+1} \, dx \leq \lambda_\alpha \|f\|_{A_\alpha^2}^2. $$
 \begin{thm} \label{thm-1.3}
 	For every $f \in a_\alpha^2$ with $f(0)=0$ and every $\zeta \in \partial\mathbb{D}$,$$\int_0^1 |f(x\zeta)|^2 (1 - x)^{\alpha+1} \, dx \leq 2\lambda_\alpha \|f\|_{a_\alpha^2}^2,$$
 	where $\lambda_{\alpha}$ as in Theorem \ref{thm-E}.
 \end{thm}
 \begin{rem}
 For $p=2$, Theorem~\ref{thm-1.2} yields
 \[
 \int_{-1}^{1}(1-|x|)^{\alpha+1}|f(x)|^2\,dx
 \le 2\pi\int_{\mathbb D}|f(z)|^2\,dA_\alpha(z),
 \]
 since
 \[
 \sec^2\!\left(\frac{\pi}{4}\right)=2.
 \]
 
 On the other hand, Theorem~\ref{thm-1.3} asserts that for every
 $f\in a_\alpha^2$ with $f(0)=0$,
 \[
 \int_0^1 |f(x\zeta)|^2(1-x)^{\alpha+1}\,dx
 \le 2\lambda_\alpha \|f\|_{a_\alpha^2}^2,
 \qquad \zeta\in\partial\mathbb D.
 \]
 Applying this inequality with $\zeta=1$ and $\zeta=-1$, and then adding the resulting inequalities, we obtain
 \[
 \int_{-1}^{1}(1-|x|)^{\alpha+1}|f(x)|^2\,dx
 \le 4\lambda_\alpha \|f\|_{a_\alpha^2}^2.
 \]
 
 Hence, the constant in Theorem \ref{thm-1.3} is smaller than that in Theorem \ref{thm-1.2} whenever $4\lambda_\alpha<2\pi$, 
 or equivalently, $\lambda_\alpha<\frac{\pi}{2}$.
 In particular, if $\alpha\ge0$, then by Theorem~\ref{thm-E},
 \[
 \lambda_\alpha\le\frac{1}{1+\alpha},
 \]
 and therefore,
 \[
 4\lambda_\alpha
 \le\frac{4}{1+\alpha}
 \le4
 <2\pi.
 \]
 Consequently, for every $\alpha\ge0$, Theorem \ref{thm-1.3} provides a better constant than that obtained in Theorem \ref{thm-1.2} when $p=2$.
 \end{rem}
 %Our second main result establishes an exact area-weight analogue of the Frazer and Riesz-Fejér directional restrictions, shifting the classical boundary-integral controls to full-domain area bounds.
\section{\textbf{Application to M\"obius invariant function spaces}}

M\"obius invariant function spaces form an important class of function spaces in complex analysis. Their defining property is that they are invariant under M\"obius transformations of the unit disk. This property makes them useful in studying geometric and analytic problems. Well-known examples include the Bloch space, BMOA, Besov spaces, Dirichlet spaces, and the $Q_s$-spaces.

In order to unify these classical function spaces, Zhu introduced the family of M\"obius invariant spaces $Q(n,p,\alpha)$ in \cite{Zhu-IJM-2007}. Let $0<p<\infty$, $\alpha>-1$, and $n\in\mathbb N$. An analytic function $f$ in $\mathbb D$ belongs to $Q(n,p,\alpha)$ if
\begin{equation}\label{Qanalytic}
	\|f\|_{Q(n,p,\alpha)}^{p}
	:=
	\sup_{a\in\mathbb D}
	\int_{\mathbb D}
	\left|(f\circ\sigma_a)^{(n)}(z)\right|^{p}
	(1-|z|^{2})^{\alpha}\,dA(z)
	<\infty,
\end{equation}
where
\[
\sigma_a(z)=\frac{a-z}{1-\overline{a}z},
\qquad a\in\mathbb D,
\]
is the Möbius automorphism of $\mathbb D$ interchanging $0$ and $a$.

Since every automorphism of the unit disk can be written as a rotation composed with $\sigma_a$, the seminorm \eqref{Qanalytic} is invariant under Möbius transformations. Consequently,
\[
f\in Q(n,p,\alpha)
\quad \mbox{is equivalent to saying}\quad
f\circ\sigma\in Q(n,p,\alpha)
\]
for every $\sigma\in{\rm Aut}(\mathbb D)$, and
\[
\|f\circ\sigma\|_{Q(n,p,\alpha)}
=
\|f\|_{Q(n,p,\alpha)}.
\]

The family $Q(n,p,\alpha)$ provides a unified framework for many classical function spaces. When $p\ge1$, the quantity $|f(0)|+\|f\|_{Q(n,p,\alpha)}$ defines a complete norm on $Q(n,p,\alpha)$, making it a Banach space. For $0<p<1$, $Q(n,p,\alpha)$ is generally not a Banach space, but it is complete as a metric space. It coincides with the Bloch space whenever $np<\alpha+1$, reduces to the Besov space when $np=\alpha+2$, and contains the classical $Q_s$-spaces and BMOA as important special cases (see \cite{Zhu-IJM-2007}). Moreover, these spaces admit elegant characterizations in terms of Carleson measures and lacunary series, making them particularly useful in operator theory and geometric function theory.

Parallel to these developments, harmonic mappings have become an important object of study because of their close connections with quasiconformal mappings, minimal surfaces, nonlinear elasticity, and geometric function theory.  

%Consequently, many classical analytic spaces have been successfully extended to the harmonic setting.

Motivated by these developments, Sun {\it et al.} \cite{Sun-Liu-Wang-PA-2026} recently have introduced the harmonic M\"obius invariant spaces $Q_h(n,p,\alpha)$. A harmonic mapping $f=h+\overline g$ belongs to $Q_h(n,p,\alpha)$ if
\begin{equation}\label{Qharmonic}
	\|f\|_{Q_h(n,p,\alpha)}^{p}
	=
	\sup_{a\in\mathbb D}
	\int_{\mathbb D}
	\left(
	\left|(h\circ\sigma_a)^{(n)}(z)\right|
	+
	\left|(g\circ\sigma_a)^{(n)}(z)\right|
	\right)^p
	(1-|z|^2)^\alpha
	\,dA(z)
	<\infty.
\end{equation}

When $n=1$, this seminorm admits the equivalent representation
\[
\|f\|_{Q_h(1,p,\alpha)}^{p}
\approx
\sup_{a\in\mathbb D}
\int_{\mathbb D}
\Lambda_{f\circ\sigma_a}(z)^p
(1-|z|^2)^\alpha\,dA(z),
\]
where
\[
\Lambda_f(z)=|f_z(z)|+|f_{\overline z}(z)|.
\]
For two nonnegative quantities $A$ and $B$, we write $A \lesssim B$ if there exists a constant $C>0$ such that
$A \leq C B$. Similarly, $A \gtrsim B$ means that $B \lesssim A$, and $A \approx B$ means that both $A \lesssim B$ and $A \gtrsim B$ hold.
\vspace{1mm}

Since
\[
\Lambda_f(z)
\le |\nabla f(z)|
\le \sqrt2\,\Lambda_f(z),
\]
the seminorm can equally be formulated in terms of the Euclidean gradient. Here, $|\nabla f(z)|:=(|f_x(z)|^2+|f_y(z)|^2)^{1/2}$.

The spaces $Q_h(n,p,\alpha)$ preserve the M\"obius invariance of the analytic theory while incorporating the geometry of harmonic mappings. They provide a common framework for studying harmonic Bloch spaces, harmonic $Q_s$-spaces, harmonic BMO-type spaces, and quasiregular harmonic mappings.

Although several important developments have recently been made concerning harmonic M\"obius invariant spaces, many fundamental problems remain open. 
\vspace{1mm}

Motivated by the classical Riesz conjugate theorem, Sun {\it et al.} recently established its harmonic analogue for the first-order M\"obius invariant space $Q_h(1,p,\alpha)$ in \cite[Theorem 3.1]{Sun-Liu-Wang-PA-2026} whenever $1+\alpha <p<2+\alpha$. 
\vspace{1mm}

To the best of our knowledge, Riesz-Fej\'er inequalities for the classes $Q(n,p,\alpha)$ and $Q_h(n,p,\alpha)$ have not previously been studied. In this section, we apply Theorems~\ref{thm-1.1} and \ref{thm-1.2} to establish Riesz conjugate theorems and Riesz--Fej\'er inequalities for these spaces.
\vspace{1mm}

We first record a useful connection between the M\"obius invariant spaces
$Q(n,p,\alpha)$ and weighted Bergman spaces in the following result. 
\begin{lem} \label{lem-2.1}
Let $0 < p < \infty$ and $\alpha>np-1$. Then $Q(n,p,\alpha)\subset A_{\alpha-np}^p$ and $Q_h(n,p,\alpha)\subset a_{\alpha-np}^p$.
\end{lem}
\begin{pf} [{\bf Proof}]
Let $f \in Q(n,p,\alpha)$. Then
\[
\|f\|_{Q(n,p,\alpha)}^p
=
\sup_{a\in\mathbb D}
\int_{\mathbb D}
\left|(f\circ\sigma_a)^{(n)}(z)\right|^p
(1-|z|^2)^\alpha\,dA(z).
\]
Taking $a=0$ and observing that $\sigma_0(z)=-z$, we obtain
\[
\begin{aligned}
	\|f\|_{Q(n,p,\alpha)}^p
	&\geq
	\int_{\mathbb D}
	\left|(f\circ\sigma_0)^{(n)}(z)\right|^p
	(1-|z|^2)^\alpha\,dA(z)\\
	&=
	\int_{\mathbb D}
	|f^{(n)}(-z)|^p
	(1-|z|^2)^\alpha\,dA(z)\\
	&=
	\int_{\mathbb D}
	|f^{(n)}(z)|^p
	(1-|z|^2)^\alpha\,dA(z),
\end{aligned}
\]
where the last equality follows from the change of variable $z\mapsto -z$.
Consequently, $f\in Q(n,p,\alpha)$ implies $f^{(n)}\in A_\alpha^p$ for $\alpha>-1$.
On the other hand, for $\beta>-1$, the standard derivative characterization of
Bergman spaces from \cite[Theorem 4.28]{Zhu-OTFS-2007} asserts that
\[
f\in A_\beta^p
\quad \mbox{if and only if} \quad
(1-|z|^2)^n f^{(n)}(z)\in L^p(\mathbb D,dA_\beta),
\]
or, equivalently,
\[
f\in A_\beta^p
\quad \mbox{reduces to} \quad
\int_{\mathbb D}
|f^{(n)}(z)|^p
(1-|z|^2)^{\beta+np}\,dA(z)<\infty.
\]
Choosing $\beta=\alpha-np$, we therefore obtain the continuous inclusion
\[
Q(n,p,\alpha)\subset A_{\alpha-np}^p,
\]
whenever the parameters are in the range $\alpha>np-1$.

An analogous observation applies to the harmonic M\"obius invariant spaces.
Indeed, let $f=h+\overline{g}\in Q_h(n,p,\alpha)$. By definition,
\[
\|f\|_{Q_h(n,p,\alpha)}^p
=
\sup_{a\in\mathbb D}
\int_{\mathbb D}
\left(
\left|(h\circ\sigma_a)^{(n)}(z)\right|
+
\left|(g\circ\sigma_a)^{(n)}(z)\right|
\right)^p
(1-|z|^2)^\alpha\,dA(z).
\]
Once again, taking $a=0$ and using $\sigma_0(z)=-z$ yields
\[
\begin{aligned}
	\|f\|_{Q_h(n,p,\alpha)}^p
	&\geq
	\int_{\mathbb D}
	\left(
	|h^{(n)}(-z)|+|g^{(n)}(-z)|
	\right)^p
	(1-|z|^2)^\alpha\,dA(z)\\
	&=
	\int_{\mathbb D}
	\left(
	|h^{(n)}(z)|+|g^{(n)}(z)|
	\right)^p
	(1-|z|^2)^\alpha\,dA(z)\\
	&=
	\int_{\mathbb D}
	\left(
	|h^{(n)}(z)|^p+|g^{(n)}(z)|^p
	\right)
	(1-|z|^2)^\alpha\,dA(z).
\end{aligned}
\]
Hence, $f\in Q_h(n,p,\alpha)$ leads to 
\[
\int_{\mathbb D}
|h^{(n)}(z)|^p
(1-|z|^2)^\alpha\,dA(z)
<\infty,
\]
and
\[
\int_{\mathbb D}
|g^{(n)}(z)|^p
(1-|z|^2)^\alpha\,dA(z)
<\infty
\]
for $\alpha>-1$. 
Thus, the corresponding derivative characterization of weighted Bergman spaces gives $h\in A_{\alpha-np}^p$ and $g\in A_{\alpha-np}^p$ whenever $\alpha>np-1$. Therefore, $f\in a_{\alpha-np}^p$. This shows that $Q_h(n,p,\alpha)\subset a_{\alpha-np}^p$ whenever $\alpha>np-1$, as desired.
\end{pf}

This shows that both the analytic and harmonic M\"obius invariant spaces
are continuously embedded in their corresponding weighted Bergman spaces.
\vspace{1mm}

As an application of Theorem \ref{thm-1.1}, together with Lemma \ref{lem-2.1}, we deduce the following Riesz conjugate theorem for the aforementioned spaces.
\begin{thm}
Let $0 < p < \infty$ and $\alpha>np-1$. 
\begin{enumerate} [label=(\roman*)]
	\item If $f =u+iv$ is a harmonic $K$-quasiregular function in $\mathbb{D}$ and $u\in Q_h(n,p,\alpha)$, then $v \in Q_h(n,p,\alpha)$.
	\item If $f =u+iv$ is analytic in $\mathbb{D}$ and $u\in Q(n,p,\alpha)$, then $v \in Q(n,p,\alpha)$.
\end{enumerate}
\end{thm}
\begin{pf}[{\bf Proof}]
For $\alpha>np-1$, by virtue of Lemma \ref{lem-2.1}, we have $Q_h(n,p,\alpha)\subset a_{\alpha-np}^p$. The conclusion now immediately follows from Theorem \ref{thm-1.1}. For the analytic case, we have $Q(n,p,\alpha)\subset A_{\alpha-np}^p$ and the result follows from \cite[Corollary 6]{Pelaez-Rattya-AMP-2020}. This concludes the proof.
\end{pf}

\begin{rem}
Alternatively, the preceding result can be obtained from the connection
between the spaces $Q(n,p,\alpha)$, $Q_h(n,p,\alpha)$, and the Bloch
space. The analytic Bloch space $\mathcal{B}$ consists of all analytic functions $f$ in $\mathbb{D}$ such that $\|f\|_{\mathcal{B}}:=|f(0)|+\sup_{z\in\mathbb{D}}(1-|z|^2)|f'(z)|<\infty$. Similarly, the harmonic Bloch space $\mathcal{B}_{h}$ consists of all harmonic functions $f$ in $\mathbb{D}$ such that $\|f\|_{\mathcal{B}_{h}}:=|f(0)|+\sup_{z\in\mathbb{D}}(1-|z|^2)|\nabla f(z)|<\infty$, where $|\nabla f(z)|=(|f_x(z)|^2+|f_y(z)|^2)^{1/2}$.

It is known that, whenever $\alpha>np-1$, $Q(n,p,\alpha)=\mathcal{B}$ and $Q_h(n,p,\alpha)=\mathcal{B}_{h}$ (see \cite{Sun-Liu-Wang-PA-2026,Zhu-IJM-2007}).

Let $u\in\mathcal{B}_h$ be real-valued, and let $v$ be its unique harmonic conjugate with $v(0)=0$. Then $F=u+iv$ is analytic in $\mathbb{D}$. By the Cauchy--Riemann equations, it follows that
\[
|\nabla u(z)|^2
=
u_x(z)^2+u_y(z)^2
=
v_x(z)^2+v_y(z)^2
=
|\nabla v(z)|^2,
\]
and therefore $|\nabla u(z)|=|\nabla v(z)|$ for $z\in\mathbb{D}$. Consequently,
\[
\begin{aligned}
	\sup_{z\in\mathbb{D}}
	(1-|z|^2)|\nabla v(z)|
	&=
	\sup_{z\in\mathbb{D}}
	(1-|z|^2)|\nabla u(z)|<\infty.
\end{aligned}
\]
Since $v(0)=0$, we further have
\[
\|v\|_{\mathcal{B}_h}
=
\sup_{z\in\mathbb{D}}
(1-|z|^2)|\nabla v(z)|
\leq
|u(0)|
+
\sup_{z\in\mathbb{D}}
(1-|z|^2)|\nabla u(z)|
=
\|u\|_{\mathcal{B}_h}.
\]
Hence, $v \in \mathcal{B}_h$. 
\vspace{1mm}

On the other hand, let
$
f=u+iv=h+\overline{g}
$
be a $K$-quasiregular mapping in $\mathbb{D}$. Then $|g'(z)|\le k|h'(z)|$, where $k=\frac{K-1}{K+1}<1
$.
Since $u=\Re(h+g)$ and $v=\Im(h-g)$,
we have
$$
|\nabla u|^2
=
|h'|^2+|g'|^2+2\Re(h'\overline{g'}),
$$
and
$$
|\nabla v|^2
=
|h'|^2+|g'|^2-2\Re(h'\overline{g'}).
$$
Using the estimate $|\Re(h'\overline{g'})|\le |h'||g'| \le k|h'|^2$, it follows that
$$
(1-k)^2|h'|^2
\le
|\nabla u|^2,\,
|\nabla v|^2
\le
(1+k)^2|h'|^2.
$$
Hence,
$$
\frac{1-k}{1+k}\,|\nabla u|
\le
|\nabla v|
\le
\frac{1+k}{1-k}\,|\nabla u|,
$$
which is equivalent to
$$
\frac1K\,|\nabla u|
\le
|\nabla v|
\le
K\,|\nabla u|.
$$
Therefore,
$$
\sup_{z\in\mathbb D}
(1-|z|^2)|\nabla v(z)|
\le
K
\sup_{z\in\mathbb D}
(1-|z|^2)|\nabla u(z)|.
$$
Consequently, if $u\in\mathcal{B}_h$, then $v\in\mathcal{B}_h$, and
$
\|v\|_{\mathcal{B}_h}
\le
K\|u\|_{\mathcal{B}_h},
$
assuming the normalization $v(0)=0$. This provides an alternative proof
of the desired result.
\end{rem}

In view of Theorem \ref{thm-1.2}, we establish a Riesz--Fej\'er inequality for the spaces $Q(n,p,\alpha)$ and $Q_h(n,p,\alpha)$.
\begin{thm} \label{thm-2.2}
Let $ 1< p < \infty$ and $\alpha>np-1$.
\begin{enumerate}[label=(\roman*)]
	\item If $f $ is a complex-valued harmonic function belonging to the space $Q_h(n,p,\alpha)$, then the following integral inequality holds:
	$$
	\int_{-1}^{1} (1 - |x|)^{\alpha+1-np} |f(x)|^p \, dx \leq  \pi \, \sec^{p}\left(\dfrac{\pi}{2p}\right)\, \|f\|^p_{a^p_{\alpha-np}}.
	$$
	\item If $f \in Q(n,p,\alpha)$, then the following integral inequality holds:
	$$
	\int_{-1}^{1} (1 - |x|)^{\alpha+1-np} |f(x)|^p \, dx \leq  \pi \,  \|f\|^p_{A^p_{\alpha-np}}.
	$$
\end{enumerate}
\end{thm}
\begin{pf}[{\bf Proof}]
The proof is an immediate consequence of Theorem \ref{thm-1.2}, Theorem \ref{thm-D}, and Lemma \ref{lem-2.1}. We therefore omit the details.
\end{pf}

Our next objective is to present another application of our main result to the space $Q_h(n,p,\alpha)$. To this end, we first note the following elementary observation. If $f=h+\overline{g}$ is a harmonic mapping in $\mathbb{D}$ with $h,g\in A_\alpha^p$, where $0<p<\infty$ and $\alpha>-1$, then clearly $f\in a_\alpha^p$. It is natural to ask whether the converse also holds. The following lemma shows that this is indeed the case.
 \begin{lem}\label{lem-2.2}
	Let $0 < p < \infty$ and $\alpha>-1$. If $f = h + \bar{g} \in a_\alpha^p$ is a harmonic mapping in $\mathbb{D}$, then $h, \,g \in A_\alpha^p$.
\end{lem}
\begin{pf} [{\bf Proof}]
	Let $f = h + \bar{g}$. Then $|h(z)| \leq |f(z)| + |g(z)|$ in $\mathbb{D}$. For $p > 0$, we have $$|h(z)|^p \leq C_p (|f(z)|^p + |g(z)|^p),$$where $C_p > 0$ is constant depending on $p$. So, it is enough to prove that $g \in A_\alpha^p$ in order to prove $h \in A_\alpha^p$ for $p > 0$. We consider the analytic function $F = h + g$ in $\mathbb{D}$. Clearly, $u := \Re(F) = \Re(f)$. This shows that $|u| \leq |f|$ in $\mathbb{D}$, and so by the assumption, $u \in a_\alpha^p$ for $ p > 0$. Applying the Riesz theorem for weighted analytic Bergman space \cite[Corollary 6]{Pelaez-Rattya-AMP-2020} to $F$, for $p>0$ we have
	$$\|F\|_{A_\alpha^p} \leq C_{p,\alpha} \|u\|_{a_\alpha^p}.$$
	Therefore, $F \in A_\alpha^p$ for $ p > 0$. 
	Consequently, $\Im(F) \in a_\alpha^p$ for $ p > 0$. Since $\Im(f) = \Im(h) - \Im(g)$ and $\Im(F) = \Im(h) + \Im(g)$, we have $$\Im(g) = \frac{\Im(F) - \Im(f)}{2}.$$
	Using the fact $|a+ib|^p \leq C_p(|a|^p+|b|^p)$, $p>0$, we deduce that 
	$$|\Im(g)|^p \leq \tilde{C}_p \cdot (|\Im(F)|^p + |\Im(f)|^p), \quad p > 0.$$
	Since $\Im(F)$,  $\Im(f) \in a_\alpha^p$, $p > 0$, the preceding inequality yields $\Im(g) \in a_\alpha^p$, $p > 0$.
	
	Define $\tilde{g} = -ig$ in $\mathbb{D}$. Then $\tilde{g}$ is also analytic in $\mathbb{D}$ and $\Re(\tilde{g}) = \Im(g)$. Since $\Im(g) \in a_\alpha^p$ for $p > 0$, we have $\Re(\tilde{g}) \in a_\alpha^p$ for $p > 0$. Applying \cite[Corollary 6]{Pelaez-Rattya-AMP-2020} once again, now to $\tilde{g}$, we deduce that $\tilde{g} \in A_\alpha^p$ for every $p>0$. Therefore, $g \in A_\alpha^p$ for $p > 0$.
	The proof is now complete.
\end{pf}

We conclude this section with the following corollary, which is an immediate consequence of Lemmas \ref{lem-2.1} and \ref{lem-2.2}.
\begin{cor}
Let $ 0< p < \infty$ and $\alpha>np-1$. If $f=h+\overline{g}\in Q_h(n,p,\alpha)$, then $h,g\in A_{\alpha-np}^p$.
\end{cor}

	\section{\textbf{Proofs of the Main Theorems and Related Results}}
	We first establish a collection of auxiliary results that are essential for the proof of Theorem \ref{thm-1.1}. With these preparatory results in hand, we then prove Theorem \ref{thm-1.1}. 
	\vspace{1mm}
	
	Unless otherwise stated, $D_r(a)$ will denote the open disk centered at $a$ with radius $r$. That is, $D_r(a):=\{z \in \mathbb{C}: |z-a|<r\}$. For a measurable set $E$, we write $A_\alpha(E) := \int_E  dA_\alpha(z)$ for the weighted area of $E$. Set $A_0(E):=A(E)$.
	\begin{thm} \label{thm-3.1}
		Let $u$ be harmonic in $\mathbb{D}$ and suppose $D_r(a) \subset \mathbb{D}$ with $|a|<1/3, r\leq 1/3$. Then for every $p > 0$ there exists a constant $C_{p,\alpha} > 0$ such that
		
		$$|u(a)|^p \le \frac{C_{p,\alpha}}{A_\alpha(D_r(a))} \int_{D_r(a)} |u(z)|^p \, dA_\alpha(z).$$
		
		Here, $C_{p,\alpha}$ depends only on $p$ and $\alpha$.
	\end{thm}
 \begin{pf}
 We begin by proving the following comparability estimates for $z \in D_r(a)$:
 \begin{enumerate}[label=(\roman*)]
 	\item $1 - |z| \approx  1 - |a|$,\\
 	\item $1 - |z|^2 \approx 1 - |a|^2$,\\
 	\item $(1 - |z|^2)^\alpha \approx (1 - |a|^2)^\alpha$ for $\alpha \in \mathbb{R} $.\\
 	 \end{enumerate}
 	 
 \noindent We first establish (i). Since $|a|<1/3$, we can see that $$1-|a|>\frac{2}{3} \implies \frac{1}{3}< \frac{1}{2}(1-|a|).$$ 
So, for any $z \in D_r(a)$, we have
$$|z-a| < r<\frac{1}{3} < \rho(1-|a|),$$ where $\rho=\frac{1}{2}$.
Using reverse triangle inequality, we get 
 %$$||z| - |a|| \leq |z - a| < \rho (1 - |a|).$$
 %Hence
 %$$|z| \leq |a| + \rho(1 - |a|) \implies 1 - |z| \geq (1 - \rho)(1 - |a|).$$
 %Similarly,
 %$$|z| \geq |a| - \rho(1 - |a|) \implies 1 - |z| \leq (1 + \rho)(1 - |a|).$$ Thus
$$(1 - \rho)(1 - |a|) \leq 1 - |z| \leq (1 + \rho)(1 - |a|).$$
This implies $1 - |z| \approx 1 - |a|$.\\
\noindent Next, we prove (ii). Observe that 
%since we have
 %$$1 - |z|^2 = (1 - |z|)(1 + |z|),$$\
 % and
 %$$1 \leq 1 + |z| \leq 2,$$
%We obtain
 $$1 - |z| \leq 1 - |z|^2 \leq 2(1 - |z|)$$
and
 $$1 - |a| \leq 1 - |a|^2 \leq 2(1 - |a|).$$
Combining these estimates with part (i) gives 
 $$1 - |z|^2 \approx 1 - |a|^2.$$
 It remains to verify (iii). For $\alpha \geq 0$, it is obvious. For $\alpha<0$, let $\alpha=-\beta$, where $\beta >0$. By (ii), there exists positive constants $K_1$ and $K_2$, such that  \[
 K_1(1-|a|^2)\le (1-|z|^2)\le K_2(1-|a|^2).
 \]
Since \(\beta>0\),
 
 \[
 (1-|z|^2)^\beta
 \le K_2^\beta (1-|a|^2)^\beta,
 \]
which implies
\[
 (1-|z|^2)^{-\beta}
 \ge
 K_2^{-\beta}(1-|a|^2)^{-\beta}.
 \]
Similarly,
 \[
 (1-|z|^2)^{-\beta}
 \le
 K_1^{-\beta}(1-|a|^2)^{-\beta}.
 \]
Thus
 \[
 K_2^{-\beta}(1-|a|^2)^{-\beta}
 \le
 (1-|z|^2)^{-\beta}
 \le
 K_1^{-\beta}(1-|a|^2)^{-\beta}.
 \]
Equivalently,
 \[
 K_2^{\alpha}(1-|a|^2)^{\alpha}
 \le
 (1-|z|^2)^{\alpha}
 \le
 K_1^{\alpha}(1-|a|^2)^{\alpha}.
 \]
 Hence, $(1 - |z|^2)^\alpha \approx (1 - |a|^2)^\alpha$ for $\alpha \in \mathbb{R} $.\\
 
 In view of (iii), there exist $C_1, C_2 > 0$ depending only on $\alpha$, such that
 \begin{equation} \label{e-3.1-a}
 	C_1 (1-|a|^2)^\alpha
 	\le
 	(1-|z|^2)^\alpha
 	\le
 	C_2 (1-|a|^2)^\alpha ,
 	\qquad z\in D_r(a).
 \end{equation}
A direct consequence of \eqref{e-3.1-a} is
 \begin{equation}\label{H-K-P1-E-01}
 	C_1 (1+\alpha)\,(1-|a|^2)^\alpha A(D_r(a))
 	\le
 	A_\alpha(D_r(a))
 	\le
 	C_2\, (1+\alpha)\, (1-|a|^2)^\alpha A(D_r(a)).
 \end{equation}
It follows from \eqref{e-3.1-a} that
 \begin{equation}
  \begin{aligned}\label{H-K-P1-E-02}
 	\int_{D_r(a)} |u(z)|^p\, dA(z)
 	&= (1+\alpha)^{-1}\,
 	\int_{D_r(a)}
 	|u(z)|^p
 	(1-|z|^2)^{-\alpha}
 	\, dA_\alpha(z)  \\
 	&\le
 	C_1^{-1}\, (1+\alpha)^{-1}\, (1-|a|^2)^{-\alpha}
 	\int_{D_r(a)}
 	|u(z)|^p\, dA_\alpha(z). 
 \end{aligned}
 \end{equation}
An application of \cite[Lemma 2]{Fefferman-Stein-Acta-1972} yields
 \begin{equation} \label{e-3.3-a}
 	|u(a)|^p
 	\le
 	\frac{C_p}{A(D_r(a))}
 	\int_{D_r(a)}
 	|u(z)|^p\, dA(z)
 \end{equation}
 for some constant $C_p>0$, depending only on $p$.
Consequently, in view of \eqref{H-K-P1-E-02} and \eqref{e-3.3-a},
 \[
 |u(a)|^p
 \le
 \frac{C_p \, C_1^{-1}\, (1+\alpha)^{-1}}{A(D_r(a))}
 (1-|a|^2)^{-\alpha}
 \int_{D_r(a)}
 |u(z)|^p\, dA_\alpha(z).
 \]
Observe from \eqref{H-K-P1-E-01} that
 \[
 \frac{A_\alpha(D_r(a))}{A(D_r(a))}
 \le
 C_2 \, (1+\alpha)\, (1-|a|^2)^\alpha,
 \]
which in turn implies
 \[
 \frac{(1-|a|^2)^{-\alpha}}{A(D_r(a))}
 \le
 \frac{C_2\,(1+\alpha)}{A_\alpha(D_r(a))}.
 \]
Therefore,
 \[
 |u(a)|^p
 \le
 \frac{C_p C_1^{-1} C_2 }{A_\alpha(D_r(a))}
 \int_{D_r(a)}
 |u(z)|^p\, dA_\alpha(z)=\frac{C_{p,\alpha}}{A_\alpha(D_r(a))}
 \int_{D_r(a)}
 |u(z)|^p\, dA_\alpha(z).
 \]
This completes the proof.
 \end{pf}
 
 The above theorem shows that the pointwise value of a harmonic function is controlled by its weighted local average over any disk $D_a(r)$ containing the point such that $r \leq 1/3,\, |a| \leq 1/3$. Geometrically, it asserts that the value of $|u(a)|^p$ cannot be significantly larger than the normalized weighted $L^p$-mass of $u$ on a neighborhood of $a$. The normalization by the weighted area $A_\alpha(D_r(a))$ makes the estimate independent of the size of the disk and is naturally adapted to the weighted measure $dA_\alpha$. In the unweighted case, the corresponding result for harmonic functions was established by Fefferman and Stein \cite{Fefferman-Stein-Acta-1972}, while the analytic version involving the weighted area can be found in \cite{Zhu-OTFS-2007}. The above theorem extends these results to harmonic functions in weighted Bergman spaces.
 \vspace{1mm}
 
 By making use of Theorem \ref{thm-3.1}, we establish the following result for harmonic quasiregular function. 
 \begin{lem} \label{lem-3.1}
 	Let $0<p<\infty$ and let
 	\(
 	f=u+iv
 	\)
 	be a harmonic $K$-quasiregular mapping in $\mathbb D$.
 	Then, there exists a constant $C_{p,\alpha}$ such that
 	\[
 	|f_z(0)|^p
 	\le \left(\frac{1}{1-k}\right)^p C_{p,\alpha}
 	\int_{\mathbb D_{\frac{1}{2}}(0)}
 	|u(z)|^p\, dA_\alpha(z).
 	\]
 \end{lem}
 
 \begin{proof}
 	Since $f$ is a harmonic mapping in the simply connected doamin $\mathbb{D}$, it can be expressed as $f = h + \overline{g}$, where $h$ and $g$ are analytic functions in $\mathbb{D}$, uniquely determined under the normalization $g(0) = 0$.
 	\vspace{0.5mm}
 	
 	Let \(F=h+g\). Clearly,	$F$ is holomorphic and $u=\Re f=\Re F$.
 	
 By the Cauchy--Riemann equations,
 \[
 |F'(z)|^2
 =u_x(z)^2+u_y(z)^2
 =|\nabla u(z)|^2.
 \]
 In particular,
 \[
 |F'(0)|=|\nabla u(0)|.
 \]
 	
 	We next estimate $|F'(0)|$. To this end, we invoke the following standard interior gradient estimate for harmonic functions (see \cite[p.17]{Pavlovic-FunctionClasses-2019}).
 	Suppose $u$ is a real-valued harmonic function in $\mathbb{D}$, and let $a \in \mathbb{D}$ and $r>0$ be such that $D_r(a)\subset \mathbb{D}$. Then there exists a constant $K$, independent of $a$ and $r$, such that
 	\[
 	|\nabla u(a)|
 	\leq
 	\frac{K}{r}
 	\sup\{|u(z)|: z\in D_r(a)\}.
 	\]
 	Taking $\varepsilon=1/2$, we have $(\varepsilon/3)\mathbb{D}\subset\mathbb{D}$. Hence, applying the above estimate with $a=0$ and $r=\varepsilon/3$, we obtain
 	\[
 	|F'(0)|
 	= |\nabla u(0)|
 	\le C
 	\sup_{|z|<\varepsilon/3}
 	|u(z)|,
 	\]
 	where $C>0$ is an absolute constant. Consequently,
 	\[
 	|F'(0)|^p
 	\le
 	C^p
 	\sup_{|z|<\varepsilon/3}|u(z)|^p.
 	\]
We now turn to estimating the quantity $\sup_{|z|<\varepsilon/3}|u(z)|^p$. For any $z$ in the disk $(\varepsilon/3)\mathbb{D}$, applying Theorem \ref{thm-3.1} for the disk $\mathbb{D}_{\frac \varepsilon 3}(z) \subset \varepsilon \mathbb{D}$, we get
 	\[
 	|u(z)|^p
 	\le
 	\frac{C_{p,\alpha,1}}
 	{A_\alpha(D_{\frac\varepsilon 3}(z))}
 	\int_{D_{\frac\varepsilon 3}(z)}
 	|u(\xi)|^p\, dA_\alpha(\xi).
 	\]
Since
 	$D_{\frac\varepsilon 3}(z)\subset \varepsilon\mathbb D,$
 	it follows from the preceding inequality that
 	\[
      |u(z)|^p
 	\le
 	\frac{C_{p,\alpha,1}}
 	{A_\alpha(D_{\frac\varepsilon 3}(z))}\,
 	\int_{\varepsilon\mathbb D}
 	|u(\xi)|^p\, dA_\alpha(\xi).
 	\]
 Now, from \eqref{H-K-P1-E-01}, we have 
 \begin{equation} \label{e-3.5-a}
 	\frac{1}{A_\alpha(D_{\frac\varepsilon 3}(z))} \leq \frac{1}{	C_1 (1+\alpha)\,(1-|z|^2)^\alpha A(D_{\frac\varepsilon 3}(z))}=\frac{1}{	C_1 (1+\alpha)\,(1-|z|^2)^\alpha \left(\frac{\varepsilon}{3}\right)^2}.
 \end{equation}
 Since $|z|<\varepsilon/3=\frac16$, we have
 \[
 \frac{35}{36}<1-|z|^2\le1.
 \]
 Therefore, there exist positive constants
 \[
 c_\alpha=\min\left\{1,\left(\frac{35}{36}\right)^\alpha\right\}
 \quad\text{and}\quad
 C_\alpha=\max\left\{1,\left(\frac{35}{36}\right)^\alpha\right\},
 \]
 depending only on $\alpha$, such that
 \[
 c_\alpha
 \le
 (1-|z|^2)^\alpha
 \le
 C_\alpha,
 \qquad |z|<\frac{\varepsilon}{3}.
 \]
Combining the above estimates together with \eqref{e-3.5-a} yields, there exists positive constant $C_{p,\alpha,2}$ such that
\[
|u(z)|^p
\le
C_{p,\alpha,2}
\,
\int_{\varepsilon\mathbb D}
|u(\xi)|^p\, dA_\alpha(\xi).
\]
Consequently, there exists positive constant $C_{p,\alpha,3}$, depending only on $p$ and $\alpha$, such that
 	\begin{equation}\label{H-K-P1-E-03}
 	|F'(0)|^p
 	\le
 	C_{p,\alpha,3}
 	\int_{\varepsilon\mathbb D}
 	|u(z)|^p\, dA_\alpha(z).
 	\end{equation}
Since $f=h+\overline{g}$ is $K$-quasiregular, the complex dilatation of $f$, defined by $\omega(z)=g'(z)/h'(z)$ satisfies $|\omega(z)|\le k$ for $z \in \mathbb{D}$, where $k=(K-1)/(K+1)$.  A simple computation shows that
 \[
 F'(z)
 = h'(z)+g'(z)
 = h'(z)\bigl(1+\omega(z)\bigr)
 \]
 and so,
 \[
 |F'(z)|
 = |h'(z)|\,|1+\omega(z)|
 \ge |h'(z)|\bigl(1-|\omega(z)|\bigr)
 \ge (1-k)|h'(z)|.
 \]
Evaluating at $z=0$ yields
 \[
 |h'(0)|
 \le \frac{|F'(0)|}{1-k}.
 \]
 Therefore,
 \begin{equation}\label{H-K-P1-E-04}
 |h'(0)|^p
 \le
 \frac{|F'(0)|^p}{(1-k)^p}.
 \end{equation}
By virtue of \eqref{H-K-P1-E-03} and  \eqref{H-K-P1-E-04}, we deduce that
 \[
 |h'(0)|^p
 \le
 \frac{C_{p,\alpha,3}}{(1-k)^p}
 \int_{\varepsilon\mathbb D}
 |u(z)|^p\, dA_\alpha(z).
 \]
Since $f_z=h'$, the desired estimate now follows, thereby completing the proof.
 \end{proof}

It is worth mentioning that Lemma \ref{lem-3.1} provides a weighted area analogue of \cite[Lemma 1]{Das-Rasila-PA-2026}.
 \vspace{1mm}

 We now prove another auxiliary result.
 \begin{lem} \label{lem-3.2}
 	Let $0 < p \le 1$ and $f = h + \overline{g}$ be a harmonic $K$-quasiregular mapping in $\mathbb{D}$. Then there exist positive constants $C_1$ and $C_2$ such that 
 	$$
 	\int_{\mathbb{D}} |f(z)|^p dA_{\alpha}(z) \leq \frac{C_1+C_2k^p}{1+\alpha} \int_{\mathbb{D}} |h'(w)|^p (1 - |w|^2)^p dA_{\alpha}(w).
 	$$
 \end{lem}
 \begin{pf}
 Since both \(h\) and \(g\) are holomorphic in \(\mathbb D\), it follows from \cite[p. 85]{Zhu-OTFS-2007} that there exist positive constants \(C_1\) and \(C_2\), such that
 	
 	\[
 		\int_{\mathbb D} |h(z)|^{p}\, dA_{\alpha}(z)
 		\le
 		\frac{C_1}{1+\alpha}
 		\int_{\mathbb D} |h'(w)|^{p}(1-|w|^{2})^{p}\, dA_{\alpha}(w).
 	\]
 	
  	and 
 	
 	\[
 	\int_D |g(z)|^{p}\, dA_{\alpha}(z)
 	\le
 	\frac{C_2}{1+\alpha}
 	\int_\mathbb{D} |g'(w)|^{p}\,(1-|w|^{2})^{p}\, dA_{\alpha}(w).
 	\]
 	
\noindent Also, since \(0<p\le 1\), we have
 \[
 |f(z)|^p
 =
 |h(z)+\overline{g(z)}|^p
 \le
 \bigl(|h(z)|+|g(z)|\bigr)^p
 \le
 |h(z)|^p+|g(z)|^p.
 \]
Therefore,
 \begin{align*}
 	\int_{\mathbb D} |f(z)|^p\, dA_\alpha(z)
 	&\le
 	\int_{\mathbb D} |h(z)|^p\, dA_\alpha(z)
 	+
 	\int_{\mathbb D} |g(z)|^p\, dA_\alpha(z)\\
 	&\le
 	\frac{C_1}{1+\alpha}
 	\int_{\mathbb D}
 	|h'(w)|^p (1-|w|^2)^p
 	\, dA_\alpha(w)\\
 	&\qquad
 	+
 	\frac{C_2}{1+\alpha}
 	\int_{\mathbb D}
 	|g'(w)|^p (1-|w|^2)^p
 	\, dA_\alpha(w).
 \end{align*}
Using the estimate
 \(
 |g'(w)|\le k\,|h'(w)|,
 \)
 we obtain $|g'(w)|^p
 \le
 k^p |h'(w)|^p$. Hence,
 \begin{align*}
 	\int_{\mathbb D} |f(z)|^p\, dA_\alpha(z)
 	&\le
 	\frac{C_1}{1+\alpha}
 	\int_{\mathbb D}
 	|h'(w)|^p (1-|w|^2)^p
 	\, dA_\alpha(w)\\
 	&\qquad
 	+
 	\frac{C_2 k^p}{1+\alpha}
 	\int_{\mathbb D}
 	|h'(w)|^p (1-|w|^2)^p
 	\, dA_\alpha(w)\\
 	&=
 	\frac{C_1+C_2k^p}{1+\alpha}
 	\int_{\mathbb D}
 	|h'(w)|^p (1-|w|^2)^p
 	\, dA_\alpha(w).
 \end{align*}
This proves the theorem.
 \end{pf}

We are now in a position to prove Theorem \ref{thm-1.1}. The proof makes use of Theorem \ref{thm-3.1} together with Lemmas \ref{lem-3.1} and \ref{lem-3.2}.
 \subsection{Proof of Theorem \ref{thm-1.1}}
 We divide the proof into two cases according to the range of $p$, namely $p>1$ and $0<p\leq 1$.\\
 	\textbf{Case 1: When $p>1$.}\\
 	Since $f = u + iv$ is a harmonic $K$-quasiregular mapping in $\mathbb{D}$ and $1 < p < \infty$, by virtue of \cite{Liu-Zhu-AdvMath-2023}, there exists constant $C_{p,K}>0$, depending only on $p$ and $K$, such that 
 	\begin{equation} \label{e-3.8}
 	M_p(r, v) \le C_{p,K} M_p(r, u).
 	\end{equation}
 Using the definition of the weighted harmonic Bergman space norm $\| \cdot \|_{a^p_\alpha}$ and rewriting the normalized area measure in polar coordinates $dA(z) = \frac{1}{\pi} r \, dr \, d\theta$, we have
 $$
 \begin{aligned}
 		\|v\|^p_{a^p_\alpha} &= (1+\alpha)\,\int_{\mathbb{D}} |v(z)|^p (1 - |z|^2)^\alpha dA(z) \\
 		&= (1+\alpha)\,\int_{r=0}^{1} \int_{\theta=0}^{2\pi} |v(r e^{i\theta})|^p (1 - r^2)^\alpha \frac{r \, dr \, d\theta}{\pi} \\
 		&= 2\,(1+\alpha)\, \int_{r=0}^{1} \left( \frac{1}{2\pi} \int_{\theta=0}^{2\pi} |v(r e^{i\theta})|^p d\theta \right) (1 - r^2)^\alpha r \, dr \\
 		&= 2\,(1+\alpha)\, \int_{0}^{1} M_p^p(r, v) (1 - r^2)^\alpha r \, dr.
 	\end{aligned}
 $$
 Substituting \eqref{e-3.8} into the preceding identity yields
 $$\begin{aligned}
 		\|v\|^p_{a^p_\alpha} &\le 2\, (1+\alpha)\, \cdot C_{p,K}^p \int_{0}^{1} M_p^p(r, u) (1 - r^2)^\alpha r \, dr \\
 		&= C_{p,K}^p\, (1+\alpha)\, \cdot \int_{r=0}^{1} \int_{\theta=0}^{2\pi} |u(r e^{i\theta})|^p (1 - r^2)^\alpha \frac{r \, dr \, d\theta}{\pi} \\
 		&= C_{p,K}^p \|u\|_{A^p_\alpha}^p.
 	\end{aligned}$$
 This is exactly the desired conclusion.\\	
 	\noindent \textbf{Case 2: When $0<p\le 1$.}\\
 	Without loss of generality, we assume $f(0) = 0$. Since $f$ is harmonic, writting $f=h+\bar{g}$, we can see that $f_z(z) = h'(z)$.
 	
 	We shall make use of Lemma \ref{lem-3.1}, which states:
 	\begin{equation} \label{e-3.9}
 	|h'(0)|^p \le \left(\frac{1}{1-k}\right)^p C_{p,\alpha} \int_{\varepsilon\mathbb{D}} |u(z)|^p\, dA_\alpha(z),
 	\end{equation}
 	where $\varepsilon=1/2$. To generalize this pointwise bound to any arbitrary point $w \in \mathbb{D}$, we utilize the conformal automorphism $\sigma_w(z) = \frac{w - z}{1 - \bar{w}z}$. We define the following transformed composition mapping
 	$$f_{\sigma_w}(z) := (f \circ \sigma_w)(z) = u(\sigma_w(z)) + i v(\sigma_w(z)) = h_{\sigma_w}(z) + \overline{g_{\sigma_w}(z)}.$$Since $\sigma_w$ is conformal, $f_{\sigma_w}$ remains a harmonic $K$-quasiregular mapping in $\mathbb{D}$. A straightforward computation, using the chain rule evaluated at $z=0$ for the analytic component, gives
 	$$
 	|h_{\sigma_w}'(0)| = |h'(w)|(1 - |w|^2).
 	$$
 	Thus, taking $h=h_{\sigma_w}$ in \eqref{e-3.9}, we infer that
 	$$
 	|h'(w)|^p (1 - |w|^2)^p \le \left(\frac{1}{1-k}\right)^p C_{p,\alpha} \int_{\mathbb{D}_{\frac{1}{2}}(0)} |u(\sigma_w(z))|^p\, dA_\alpha(z).
 	$$
 Integrating both sides of the preceding inequality over $\mathbb{D}$ against $dA_\alpha(w)$ leads to
 \begin{equation}\label{H-K-P1-E-05}
 \int_{\mathbb{D}} |h'(w)|^p (1 - |w|^2)^p dA_\alpha(w) \le \left(\frac{1}{1-k}\right)^p C_{p,\alpha} \int_{\mathbb{D}} \left( \int_{\mathbb{D}_{\frac{1}{2}}(0)} |u(\sigma_w(z))|^p\, dA_\alpha(z) \right) dA_\alpha(w).
 \end{equation}
 Fix $w \in \mathbb{D}$ and set $\zeta = \sigma_w(z).$ Since $\sigma_w$ is an involution, $z = \sigma_w(\zeta).$ The standard identities are
 $$
 1 - |\sigma_w(\zeta)|^2 = \frac{(1 - |w|^2)(1 - |\zeta|^2)}{|1 - \bar{w}\zeta|^2}
 $$
 
  and
 $$
 |\sigma_w'(\zeta)|^2 = \frac{(1 - |w|^2)^2}{|1 - \bar{w}\zeta|^4}.
 $$
Hence
$$dA_\alpha(z) = (1+\alpha)\,(1 - |\sigma_w(\zeta)|^2)^\alpha |\sigma_w'(\zeta)|^2 \, dA(\zeta)=(1+\alpha)\,\frac{(1 - |w|^2)^{\alpha+2} (1 - |\zeta|^2)^\alpha}{|1 - \bar{w}\zeta|^{2\alpha+4}} \, dA(\zeta).
$$
Therefore
$$\int_{\mathbb{D}_{\frac{1}{2}}(0)} |u(\sigma_w(z))|^p \, dA_\alpha(z)=(1+\alpha)\,(1 - |w|^2)^{\alpha+2} \int_{D_\sigma(w, \frac{1}{2})} |u(\zeta)|^p \frac{(1 - |\zeta|^2)^\alpha}{|1 - \bar{w}\zeta|^{2\alpha+4}} \, dA(\zeta),
$$ 
 where$$D_\sigma\left(w, \frac{1}{2}\right) := \left\{\zeta : \sigma_w(\zeta) < \frac{1}{2}\right\}.$$
Substituting this into \eqref{H-K-P1-E-05}, we get 
\begin{align*}
	\int_{\mathbb{D}} |h'(w)|^p & (1 - |w|^2)^p dA_\alpha(w) \\ 
	&= \left(\frac{1}{1-k}\right)^p \tilde{C}_{p,\alpha} \int_{\mathbb{D}} \int_{D_\sigma(w, \frac{1}{2})} |u(\zeta)|^p \frac{(1 - |w|^2)^{2\alpha+2} (1 - |\zeta|^2)^\alpha}{|1 - \bar{w}\zeta|^{2\alpha+4}} \, dA(\zeta) \, dA(w).
\end{align*}
Applying Tonelli's theorem, we deduce that
\begin{align}\label{H-K-P1-E-06}
\int_{\mathbb{D}} & \notag |h'(w)|^p  (1 - |w|^2)^p dA_\alpha(w) \\ 
& = \left(\frac{1}{1-k}\right)^p \tilde{C}_{p,\alpha} \int_{\mathbb{D}} |u(\zeta)|^p (1 - |\zeta|^2)^\alpha \left[ \int_{D_\sigma(\zeta, \frac{1}{2})} \frac{(1 - |w|^2)^{2\alpha+2}}{|1 - \bar{w}\zeta|^{2\alpha+4}} \, dA(w) \right] dA(\zeta).
\end{align}

If $w \in D_\sigma\left(\zeta, \frac{1}{2}\right)$, then the standard pseudo hyperbolic comparability relations hold (see \cite[p. 69]{Zhu-OTFS-2007}), \textit{i.e.}, $1 - |w|^2 \approx 1 - |\zeta|^2$ and $|1 - \bar{w}\zeta| \approx 1 - |\zeta|^2$. So, we can write, 
$$
\frac{(1 - |w|^2)^{2\alpha+2}}{|1 - \bar{w}\zeta|^{2\alpha+4}} \le C_1 \frac{1}{(1 - |\zeta|^2)^2},
$$where $C_1$ is some positive constant. Also, from \cite[p. 69]{Zhu-OTFS-2007}, $$A \left(D_\sigma\left(\zeta, \frac{1}{2}\right)\right) \leq C_2 (1-|\zeta|^2)^2$$for some positive constants $C_2$.
 Therefore,
 $$\int_{D_\sigma(\zeta, \frac{1}{2})} \frac{(1 - |w|^2)^{2\alpha+2}}{|1 - \bar{w}\zeta|^{2\alpha+4}} \, dA(w) \le \frac{A \left(D_\sigma\left(\zeta, \frac{1}{2}\right)\right)}{(1 - |\zeta|^2)^2} \le C.$$
Now, putting these estimates in \eqref{H-K-P1-E-06}, we have
\begin{align*}
\int_{\mathbb{D}}  |h'(w)|^p  (1 - |w|^2)^p dA_\alpha(w) \leq & \left(\frac{1}{1-k}\right)^p C_{p,\alpha}' \int_{\mathbb{D}} |u(\zeta)|^p (1 - |\zeta|^2)^\alpha dA(\zeta) \\
\leq  & \left(\frac{1}{1-k}\right)^p C_{p,\alpha}'' ||u||_{a_{p}^\alpha}^p.
\end{align*}
By making use of Lemma \ref{lem-3.2} together with the preceding inequality, we arrive at
$$
\|f\|_{a^p_\alpha}^p = \int_{\mathbb{D}} |f(z)|^p dA_{\alpha}(z) \le \left( \frac{C_1+C_2k^p}{1+\alpha} \cdot \left(\frac{1}{1-k}\right)^p C_{p,\alpha}'' \right) \|u\|_{a^p_\alpha}^p.
$$Since $|v(z)| \le |f(z)|$ holds pointwise everywhere for the imaginary component, it directly follows that
$$\|v\|_{a^p_\alpha}^p = \int_{\mathbb{D}} |v(z)|^p dA_{\alpha}(z) \le \int_{\mathbb{D}} |f(z)|^p dA_{\alpha}(z) \le C^p_{p,K,\alpha} \|u\|_{a^p_\alpha}^p,
$$
where the constant $C^p_{p,K,\alpha}>0$, depends only on $p,K$, and $\alpha$.
Taking the $p$-th root on both sides leads to $\|v\|_{a^p_\alpha} \le C_{p,K,\alpha} \|u\|_{a^p_\alpha}$, which is precisely the desired estimate.
 
 \subsection{Proof of Theorem \ref{thm-1.2}}

	Let $f \in a_{\alpha}^p$ be given by its canonical decomposition $f(z) = h(z) + \overline{g(z)}$, where $h$ and $g$ are holomorphic functions on the open unit disk $\mathbb{D}$. Since $f$ is harmonic in $\mathbb{D}$, it is continuous in $\mathbb{D}$. Fix $r\in[0,1)$. Then the closed disk $\overline{\mathbb{D}}_r=\{z\in\mathbb{C}:|z|\le r\}$	is compact and contained in $\mathbb{D}$. Hence, by continuity, $f$ is bounded on $\overline{\mathbb{D}}_r$. Let
	\[
	M_r:=\max_{|z|\le r}|f(z)|<\infty.
	\]
	
	Now define the dilated function
	\[
	f_r(z):=f(rz), \qquad z\in\mathbb{D}.
	\]
	Since the composition of a harmonic function with a dilation is harmonic, $f_r$ is harmonic in $\mathbb{D}$.
	
	To show that $f_r\in h^p$, let $0<\rho<1$. Since $|r\rho e^{i\theta}|=r\rho<r$,
	we have
	\[
	|f_r(\rho e^{i\theta})|
	=
	|f(r\rho e^{i\theta})|
	\le
	M_r
	\]
	for every $\theta\in[0,2\pi]$. Therefore,
	\[
	\frac{1}{2\pi}\int_0^{2\pi}
	|f_r(\rho e^{i\theta})|^p\,d\theta
	=
	\frac{1}{2\pi}\int_0^{2\pi}
	|f(r\rho e^{i\theta})|^p\,d\theta
	\le
	M_r^p.
	\]
	Taking the supremum over $0<\rho<1$, we obtain
	\[
	\|f_r\|_{h^p}^p
	=
	\sup_{0<\rho<1}
	\frac{1}{2\pi}\int_0^{2\pi}
	|f_r(\rho e^{i\theta})|^p\,d\theta
	\le
	M_r^p
	<
	\infty.
	\]
	Hence, $f_r\in h^p$.
		
An application of \cite[Theorem 1.1]{Melentijevic-Bozin-PA-2021} to $f_r$ reveals that
\begin{equation*}
\int_{-1}^{1} |f(rx\, e^{it})|^p \, dx \leq C_p \int_{0}^{2\pi} |f(r e^{i\theta})|^p \, d\theta
\end{equation*}
	for all $t \in \mathbb{R}$, where $C_p=\dfrac{1}{2}\sec^{p}\!\left(\dfrac{\pi}{2p}\right)$ for $1<p<\infty$.
%\[	C_p=	\begin{cases}
%\dfrac{1}{2}\sec^{p}\!\left(\dfrac{\pi}{2p}\right), & \text{if } 1<p\le 2,\\[1.2ex]
%1, & \text{if } p\ge 2.
%\end{cases}
%\]
 We multiply both sides of the above inequality by $\frac{(1+\alpha)}{\pi} r (1 - r^2)^{\alpha}$ and integrate with respect to $r$ from $0$ to $1$, we obtain
 \begin{equation} \label{e-3.12-a}
 	\frac{(1+\alpha)}{\pi} \int_{0}^{1} r(1 - r^2)^{\alpha} \left( \int_{-1}^{1} |f(rx\,e^{it})|^p \, dx \right) dr \leq C_p \int_{\mathbb{D}} |f(z)|^p \, dA_{\alpha}(z) .
 \end{equation}
By substituting $u = rx$ to the left-hand side of the above inequality, we get
 $$\frac{(1+\alpha)}{\pi} \int_{0}^{1} (1 - r^2)^{\alpha} \left( \int_{-r}^{r} |f(u)|^p \, du \right) dr.
 $$
By using Fubini's theorem, we rewrite this as
\begin{equation} \label{e-3.13}
	\frac{ (1+\alpha)}{\pi} \int_{-1}^{1} |f(u)|^p \left( \int_{|u|}^{1} (1 - r^2)^{\alpha} \, dr \right) du.
\end{equation}
We shall now estimate the inside integral. Let
$$
I:=\int_{|u|}^{1} (1 - r^2)^{\alpha} \, dr.
$$ 
Using the fact $(1-r^2)^\alpha>r\,(1-r^2)^\alpha$, we see that
$$
I\geq  \int_{|u|}^{1} (1 - r^2)^{\alpha} \,r\, dr = \frac{1}{2\,(1+\alpha)} (1 - u^2)^{\alpha+1}.
$$
Since $(1-u^2)^{1+\alpha} \geq (1-|u|)^{1+\alpha}$ for $u \in [0,1]$, using the above estimate, from \eqref{e-3.12-a} and \eqref{e-3.13}, we deduce that 
$$\int_{-1}^{1} (1 - |u|)^{\alpha+1} |f(u)|^p \, du \leq 2\, \pi \, C_p \int_{\mathbb{D}} |f(z)|^p \, dA_{\alpha}(z).$$
This concludes the proof.

 \subsection{Proof of Theorem \ref{thm-1.3}}

 	For $\zeta \in \partial\mathbb{D}$, define $f_\zeta(w) = f(\zeta w)$. Since the change of variables $z = \zeta w$ has unit Jacobian and preserves $|z|$, the measure $dA_\alpha$ is invariant under this rotation. Consequently, $f_\zeta \in a_\alpha^2$ with $\|f_\zeta\|_{a_\alpha^2} = \|f\|_{a_\alpha^2}$. Since $f_\zeta(x) = f(x\zeta)$ for $x \in (0,1)$, it suffices to establish the inequality at $\zeta = 1$ and apply it to $f_\zeta$. Henceforth we assume $\zeta = 1$.\\[2mm]
 	Every $f \in a_\alpha^2$ admits a unique harmonic expansion of the form 
 	$$f(z) = \sum_{n=1}^\infty a_n z^n + \sum_{n=1}^\infty b_n \bar{z}^n. 
 	$$
 Since $dA_\alpha$ is radial, the family $\{1, z^n, \bar{z}^n : n \in \mathbb{N}\}$ is orthogonal in $L^2(\mathbb{D}, dA_\alpha)$. Indeed, for $n \geq 0$ and $m \geq 1$
 $$
 \int_\mathbb{D} z^n \overline{\bar{z}^m} \, dA_\alpha(z) =(1+\alpha)\, \int_0^1 r^{n+m+1}(1-r^2)^\alpha \, dr \int_0^{2\pi} \frac{e^{i(n+m)\theta}}{\pi} \, d\theta = 0,
 $$
 since $n+m \geq 1$. The notation $L^2(\mathbb{D}, dA_\alpha)$ is adopted from Zhu's book \cite{Zhu-OTFS-2007}. Setting $\beta_n = \|z^n\|_{A_\alpha^2}^2$, a polar-coordinate computation with the Beta function yields
 $$\beta_n = \frac{\Gamma(n+1) \, \Gamma(\alpha+2)}{\Gamma(n+\alpha+2)}, \quad n \geq 0,$$
 and the Parseval-type identity 
 $$\|f\|_{a_\alpha^2}^2 = \sum_{n=1}^\infty \beta_n (|a_n|^2 + |b_n|^2).$$
Write $f = f_1 + if_2$, where$$f_1 = \frac{1}{2}(f + \bar{f}), \quad f_2 = \frac{1}{2i}(f - \bar{f})$$are real-valued harmonic functions on $\mathbb{D}$. Because $f_1, f_2$ are real-valued, $|f(z)|^2 = f_1(z)^2 + f_2(z)^2$ pointwise, and therefore
$$\|f\|_{a_\alpha^2}^2 = \|f_1\|_{a_\alpha^2}^2 + \|f_2\|_{a_\alpha^2}^2. $$
It thus suffices to prove the inequality, with the constant $2\, \lambda_{\alpha}$, for an arbitrary real-valued $f \in a_\alpha^2$. 

Assume $f$ is real-valued. Then $\bar{f}(z) = f(z)$ forces $b_n = \bar{a}_n$ in the expansion of $f$, so that
$$
f(z) = \sum_{n=1}^\infty a_n z^n + \sum_{n=1}^\infty \bar{a}_n \bar{z}^n. $$
Define the holomorphic lift
$$F(z) = 2\sum_{n=1}^\infty a_n z^n, \quad z \in \mathbb{D}. $$ 

For $x \in (0,1) \subset \mathbb{R}$, we have $$f(x) =  2\sum_{n=1}^\infty \text{Re}(a_n) \, x^n = \text{Re } F(x),$$
Hence
\begin{equation} \label{e-3.14}
	|f(x)| \leq |F(x)| \quad \text{for every } x \in (0,1).
\end{equation}
Now, if $b_n = \bar{a}_n$, we have $$\|f\|_{a_\alpha^2}^2 =  2\sum_{n=1}^\infty \beta_n |a_n|^2,$$while, by the orthogonality of $\{z^n\}$ in $A_\alpha^2$,$$\|F\|_{A_\alpha^2}^2 =  4\sum_{n=1}^\infty \beta_n |a_n|^2.$$
Comparing the two expressions, $\|F\|_{A_\alpha^2}^2 =  2\|f\|_{a_\alpha^2}^2$. In particular, $F \in A_\alpha^2$. Applying \cite[Theorem 1]{Andreev-Palermo-2012} to $F$ and then from \eqref{e-3.14}, we obtain
 $$\int_0^1 |f(x)|^2 (1 - x)^{\alpha+1} \, dx \leq \int_0^1 |F(x)|^2 (1 - x)^{\alpha+1} \, dx \leq \lambda_\alpha \|F\|_{A_\alpha^2}^2 = 2\lambda_\alpha \|f\|_{a_\alpha^2}^2. $$
 
 For a general complex-valued $f \in a_\alpha^2$, decompose $f = f_1 + if_2$ and using above inequality for each real-valued component, we have
 $$\int_0^1 |f_j(x)|^2 (1 - x)^{\alpha+1} \, dx \leq 2\lambda_\alpha \|f_j\|_{a_\alpha^2}^2, \quad j=1,2.$$
 Adding these inequalities and using identity $|f|^2 = f_1^2 + f_2^2$, we get
 $$\int_0^1 |f(x)|^2 (1 - x)^{\alpha+1} \, dx \leq 2\lambda_\alpha \left(\|f_1\|_{a_\alpha^2}^2 + \|f_2\|_{a_\alpha^2}^2\right) = 2\lambda_\alpha \|f\|_{a_\alpha^2}^2.
 $$
 Hence, for all $\zeta \in \partial \mathbb{D}$, the desired inequality follows, thereby completing the proof.

%	\noindent\textbf{Acknowledgement:}  The second named author thank CSIR for the support. 
	\vspace{1.5mm}
	
	\noindent\textbf{Compliance of Ethical Standards:}\\
	\noindent\textbf{Conflict of interest.} The authors declare that there is no conflict  of interest regarding the publication of this paper.
	\vspace{1.5mm}
	
	\noindent\textbf{Data availability statement.}  Data sharing is not applicable to this article as no datasets were generated or analyzed during the current study.\vspace{1.5mm}
	
	\noindent\textbf{Authors contributions.} The authors are listed in alphabetical order by surname. Both authors made equal contributions to the research, the derivation of the results, and the writing and preparation of the manuscript.

\end{document}